\documentclass[11pt]{amsart}

\usepackage{enumerate, amsmath, amsthm, amsfonts, amssymb, xy,  mathrsfs, graphicx, paralist, lscape}
\usepackage[usenames, dvipsnames]{color}
\usepackage{fullpage}
\usepackage{comment}
\usepackage{tikz}
\usepackage{subfig}
\usepackage{bm}
\usetikzlibrary{matrix,arrows}
\usepackage[normalem]{ulem}

\numberwithin{equation}{section}
\newtheorem{theorem}[equation]{Theorem}

\newtheorem{proposition}[equation]{Proposition}
\newtheorem{lemma}[equation]{Lemma}
\newtheorem{corollary}[equation]{Corollary}

\theoremstyle{definition}
\newtheorem{rmk}[equation]{Remark}
\newenvironment{remark}[1][]{\begin{rmk}[#1] \pushQED{\qed}}{\popQED \end{rmk}}
\newtheorem{rmks}[equation]{Remarks}

\newtheorem{eg}[equation]{Example}
\newenvironment{example}[1][]{\begin{eg}[#1] \pushQED{\qed}}{\popQED \end{eg}}
\newtheorem{defn}[equation]{Definition}
\newenvironment{definition}[1][]{\begin{defn}[#1]\pushQED{\qed}}{\popQED \end{defn}}
\newtheorem{altdefn}[equation]{Alternative Definition}

\newtheorem{ques}[equation]{Question}

\newtheorem{notn}[equation]{Notation}

\usepackage[colorlinks=true, pdfstartview=FitV, linkcolor=blue, citecolor=blue, urlcolor=blue]{hyperref}

\renewcommand{\phi}{\varphi}
\renewcommand{\emptyset}{\varnothing}

\newcommand{\init}{\mathrm{init}}

\makeatletter
\def\Ddots{\mathinner{\mkern1mu\raise\p@
\vbox{\kern7\p@\hbox{.}}\mkern2mu
\raise4\p@\hbox{.}\mkern2mu\raise7\p@\hbox{.}\mkern1mu}}
\makeatother

\DeclareMathOperator{\height}{height}

\usetikzlibrary{calc}
\include{pipe-macros-stumpf}

\title{A note on toric ideals of graphs and Knutson-Miller-Yong Decompositions}
\author{Sergio Da Silva}
\author{Emma Naguit}
\author{Jenna Rajchgot}

\address[Sergio Da Silva]{
    Dept.\ of Mathematics and Economics, 
    Virginia State University, 
    1 Hayden Drive,
    Petersburg, Virginia 23806, USA}
\email{sdasilva@vsu.edu, smd322@cornell.edu}

\address[Emma Naguit]{
    Dept.\ of Mathematics and Statistics, 
    McMaster University, 
    1280 Main Street West, 
    Hamilton, Ontario L8S 4K1, Canada}
\email{naguite@mcmaster.ca}

\address[Jenna Rajchgot]{
    Dept.\ of Mathematics and Statistics, 
    McMaster University, 
    1280 Main Street West, 
    Hamilton, Ontario L8S 4K1, Canada}
\email{rajchgoj@mcmaster.ca}

\keywords{toric ideals of graphs, chromatic number, geometric vertex decomposition} 
\subjclass[2020]{Primary: 13P10, 14M25 ; Secondary:  05C15, 05E40, 05C25}

\begin{document}

\begin{abstract}
We use a Gr\"obner basis technique first introduced by Knutson, Miller and Yong to study the interplay between properties of a graph $G$ and algebraic properties of the toric ideal that it defines. We first recover a well-known height formula for the toric ideal of a graph $I_G$ and demonstrate an algebraic property that can detect when a graph deletion is bipartite. We also  bound the chromatic number $\chi(G)$ using information about an  initial ideal of $I_G$. 
\end{abstract}
\maketitle

Knutson, Miller and Yong introduced the notion of a geometric vertex decomposition in \cite{KMY} to study diagonal degenerations of Schubert varieties. It has since been generalized by the third author and Klein in \cite{KR} to geometric vertex decomposability, a concept which itself generalizes vertex decomposability of simplicial complexes associated to squarefree monomial ideals via the Stanley-Reisner correspondence. They also showed that geometric vertex decompositions are related to elementary $G$-biliaisons from liaison theory.

Toric ideals of graphs are especially amenable to the application of these decompositions because of the convenient description of their generating set using Gr\"obner bases. For example, in \cite{CDSRVT}, the geometric vertex decomposability of certain families of toric ideals of graphs was explored, and a framework for studying toric ideals of graphs with squarefree degenerations was established. Indeed, in this setting, the combinatorics of a graph $G$ can inform the algebraic properties of $I_G$. This viewpoint has been well-studied (see \cite{DAli},\cite{GHKKPVT}, \cite{GRV}, \cite{OH}, and  \cite{TT}, to name a few). It is natural to ask whether the reverse is also true. Can we detect properties of $G$ from algebraic properties of $I_G$?

In this short note on the topic, we study a particular decomposition of an initial ideal first studied by Knutson, Miller and Yong \cite{KMY} which we refer to as a Knutson-Miller-Yong (KMY) decomposition. This decomposition is a generalization of a geometric vertex decomposition introduced in the same paper. In this article, we use KMY decompositions to study the chromatic number of a graph $G$. We also recover the previously known height formula for toric ideals of graphs. Although the graph theory results in this paper are not new, one of the main goals of the present work is to provide examples where commutative-algebraic techniques and toric ideals can be used to inform graph theory (rather than the other way around).

We refer the reader to \cite{CAH,HR} and \cite{Sturmfels,Villa} for the relevant background on graphs and toric ideals of graphs, respectively. Let $G = (V(G),E(G))$ be a finite simple graph with 
vertex set $V(G) =\{v_1,\ldots,v_n\}$ and edge
set $E(G) = \{e_1,\ldots,e_t\}$ where each $e_i = \{v_{j_i},v_{k_i}\}$.
Let $\mathbb{K}[E(G)] = \mathbb{K}[e_1,\ldots,e_t]$ be a polynomial
ring, where we treat the $e_i$'s as indeterminates.  Similarly,
let $\mathbb{K}[V(G)] = \mathbb{K}[v_1,\ldots,v_n]$.  Then the \emph{toric ideal of $G$}, denoted by $I_G$, is the kernel of the $\mathbb{K}$-algebra homomorphism $\varphi_G:\mathbb{K}[E(G)] \rightarrow \mathbb{K}[V(G)]$ defined by
\[\varphi_G(e_i) = v_{j_i}v_{k_i}  ~~\mbox{where $e_i = \{v_{j_i},v_{k_i}\}$ for
all $i \in \{1,\ldots,t\}$}.\]

A \emph{closed even walk} of $G$ is a sequence of vertices and edges \[\{v_{i_0},e_{i_1},v_{i_1},e_{i_2},\cdots,e_{i_k},
v_{i_k}\}\] 
such that $e_{i_j} = \{v_{i_{j-1}},v_{i_j}\}$, $v_{i_k} = v_{i_0}$, and $k$ is even.
Note that $e_{i_1}e_{i_3}\cdots e_{i_{2\ell-1}} - 
e_{i_2}e_{i_4}\cdots e_{i_{2\ell}}$ is a binomial, and by \cite[Proposition 10.1.5]{Villa}, $I_G$ is generated by the set of binomials of this form for all closed even walks of $G$. A binomial $u-v \in I_G$ is {\it primitive} if there is no other binomial $u'-v' \in I_G$ such that $u'|u$ and $v'|v$. By \cite[Proposition 10.1.9]{Villa}, the set of primitive binomials of $I_G$ which also correspond to closed even walks of $G$ define a universal Gr\"obner basis of $I_G$. Cycles are closed walks where all vertices are distinct, except for the endpoints. We will say that an edge $e$ belongs to a cycle if it appears in the list of vertices and edges defining the walk. In general, given a primitive binomial in $I_G$, we will refer to the subgraph of the binomial as the subgraph of $G$ consisting of the vertices and edges appearing in the walk.

Using KMY decompositions, we are able to recover a known result about the height of a toric ideal $I_G$ (see \cite{Villa} for example). Recall that a graph $G$ is {\it connected} if for any two pairs of vertices in $G$, there is a walk in $G$ between them.

\begin{theorem}\cite[Lem. 3.1 and Prop. 3.2]{Villa2}
Let $G$ be a connected graph with $p$ vertices, $q$ edges, and toric ideal $I_G\subset \mathbb{K}[E(G)]$. Then
\begin{equation*}
    \height(I_G)=
    \begin{cases}
      q-p & \text{if $G$ is not bipartite,}\\
      q-p+1 & \text{if $G$ is bipartite.}
    \end{cases}
  \end{equation*}
\end{theorem}

A natural induction technique in graph theory is to use an edge deletion to reduce the number of edges in a graph. One graph property that can be detected in a graph deletion using KMY decompositions is the bipartite property. Recall that the chromatic number $\chi(G)$ of a graph $G$ is the minimum number of colors needed to color the vertices of a graph so that no two adjacent vertices receive the same color. A graph $G$ is \emph{bipartite} if its chromatic number satisfies $\chi(G)\leq 2$. The next result is a combination of Proposition \ref{prop:notBipartiteGVD} and Lemmas \ref{lemma:N ideal deletion ideal} and \ref{lemma: not trivial or bridge}.

\begin{theorem}
Suppose that $G$ is a non-bipartite finite simple graph and $y\in E(G)$ defines a nondegenerate KMY decomposition of $I_G$. Then $G\setminus y$ is not bipartite. Furthermore, if $G$ is connected, then there exists a sequence of edges $e_1,\ldots, e_k$ such that $G_i:=G\setminus \{e_1,\ldots, e_i\}$ (ie. the graph defined by deleting the edges $e_1,\ldots,e_i$ from $G$) is connected and non-bipartite for $1\leq i\leq k$, $I_{G_k}=\langle 0\rangle$, and $\height(I_G)=k$.
\end{theorem}

Finally, we can also use KMY decompositions to provide a bound on the general chromatic number of a graph. This is discussed further in Section 3 and summarized in the next theorem.

\begin{theorem}
Let $G$ be a finite simple connected graph defining the toric ideal $I_G \subset R=\mathbb{K}[e_1,\ldots,e_n]$ and let $<$ be a monomial order on $R$. Suppose that $\mathcal{E}\subseteq \{e_1,\ldots,e_n\}$ is a (possibly empty) minimal set of variables such that $\init_<(I_G)\subseteq \langle \mathcal{E}\rangle$. Then 
\[\chi(G)\leq |\mathcal{E}| +3.\]
\end{theorem}

\noindent It is worth noting that Gr\"obner basis techniques have been used to study $\chi(G)$ in other contexts, such as in \cite{HW} where the chromatic number of a graph is related to whether a Gr\"obner basis of an $r$-coloring ideal has a certain form. There have also been other explicit algebraic descriptions of the chromatic number in terms of powers of the cover ideal of a graph \cite{FHT}. Our bound provides an alternate description from the perspective of toric ideals of graphs.

\subsection*{Acknowledgements} 
This work is based on the thesis project of the second author in \cite{Naguit}. The first author was supported in part by an NSERC Postdoctoral Fellowship of Canada 546010-2020. The second author was supported in part by an NSERC USRA and the third author's grant. The third author was supported in part by NSERC Discovery Grant 2023-04800. We thank the anonymous referee for helpful comments.

\section{KMY Decompositions and Graph-theoretic Properties}

In this section we will explore an ideal decomposition introduced by Knutson, Miller and Yong in \cite{KMY} and its relationship to toric ideals of graphs. 

Let $\mathbb{K}$ be an arbitrary field and $R=\mathbb{K}[x_1,\ldots,x_n]$ be the polynomial ring over $\mathbb{K}$ in $n$ indeterminates. Fix $y = x_j$ for some $1\le j \le n$. We define the \textbf{initial $y$-form} of a polynomial $f\in R$ to be the sum of all terms of $f$ having the highest power of $y$, and we denote this by $\init_yf$. More specifically, for a nonzero polynomial $f\in R$, we can write it as the sum $f = \sum_{i=0}^{N}\ {\textbf{a}_iy^i}$ where each $\textbf{a}_i \in \mathbb{K}[x_1,\ldots,\hat{y},\ldots,x_n]$ and $\textbf{a}_N \neq 0$, so that $\init_yf = \textbf{a}_N y^N$. Note that $\mathbb{K}[x_1,\ldots,\hat{y},\ldots,x_n]$ denotes the polynomial ring over $\mathbb{K}$ defined by all indeterminates $x_1,\ldots,x_n$ except $y$. For an ideal $I \subseteq R$, we define \[\init_yI := \langle \init_yf \mid f \in I \rangle.\] We say that a monomial order is \textbf{$y$-compatible} if $\init_{<}f = \init_{<}(\init_yf)$ for every $f\in R$. We refer the reader to \cite{CLO,EH} for the relevant background on Gr\"obner bases.

\begin{definition}\label{def:KMYD}
Let $I \subseteq R$ be an ideal and $<$ be a $y$-compatible monomial order. Let $G = \{ y^{d_i}q_i + r_i \mid 1 \le i \le m \}$ be a Gr\"obner basis of $I$ where $y$ does not divide any $q_i$ and $\init_y(y^{d_i}q_i + r_i) = y^{d_i}q_i$. 
Let us define two ideals of $R$: \[C_{y,I} := \langle q_i \mid 1 \le i \le m \rangle \hspace{1cm} N_{y,I} := \langle q_i \mid d_i = 0 \rangle.\] Then by \cite[Theorem 2.1]{KMY}, we necessarily have that $\sqrt{\init_yI} = \sqrt{C_{y,I}} \cap \sqrt{N_{y,I} + \langle y \rangle}$. We call this decomposition of $\sqrt{\init_y I}$ a \textbf{Knutson-Miller-Yong decomposition of} \bm{$I$} \textbf{with respect to} \bm{$y$} (or a KMY decomposition for short). It also follows from \cite[Theorem 2.1]{KMY} that $\init_{y}I = \langle y^{d_i}q_i \mid 1 \le i \le m \rangle$ and the given generating sets for $C_{y,I}, N_{y,I}$ are Gr\"obner bases for these ideals \cite[Theorem 2.1]{KMY}.
\end{definition}

\begin{remark}
If we have that $\init_yI = C_{y,I} \cap (N_{y,I} + \langle y \rangle)$, then this decomposition is called a \textbf{geometric vertex decomposition} of $I$ with respect to $y$, as defined by Knutson, Miller and Yong in \cite{KMY}. This notion was later used by Klein and Rajchgot in \cite{KR} to establish a connection to liaison theory. The definitions and results that follow in this section were originally defined in the context of geometric vertex decompositions rather than Knutson-Miller-Yong decompositions.
\end{remark}

\subsection{Properties of KMY Decompositions}

We start by defining nondegenerate KMY decompositions which will be useful in the context of toric ideals of graphs.

\begin{definition}
Let $I$ and $y$ be as defined in Definition \ref{def:KMYD}. A Knutson-Miller-Yong decomposition is \textbf{degenerate} if $\sqrt{C_{y,I}} = \sqrt{N_{y,I}}$ or if $C_{y,I} = \langle 1 \rangle$. It is \textbf{nondegenerate} otherwise.
\end{definition}

For the decompositions of toric ideals of graphs that we will be considering, the case where $C_{y,I}=\langle 1 \rangle$ doesn't occur, so we focus on characterizing the other instance of degenerate KMY decompositions. The next result is a slight adaptation of \cite[Proposition 2.4]{KR} to the language of KMY decompositions. Specifically, we have replaced the term ``geometric vertex decomposition'' in \cite[Proposition 2.4]{KR} with ``KMY decomposition'' in Proposition \ref{lemma: new KR 2.4} below. Nevertheless, the proof of Proposition \ref{lemma: new KR 2.4} is identical to the proof of \cite[Proposition 2.4]{KR} and so is omitted.

\begin{proposition} \label{lemma: new KR 2.4}
\cite[Proposition 2.4]{KR} Let $I$ be a radical ideal of $R$ and let $<$ be a $y$-compatible monomial order on $R$ and suppose that $\sqrt{C_{y,I}}=\sqrt{N_{y,I}}$ (so that $y$ defines a degenerate KMY decomposition). Then the reduced Gr\"obner basis of $I$ does not involve $y$, and $I = \init_{y} I = C_{y,I} = N_{y,I}$.
\end{proposition}

\begin{remark}
The theorem is not true if $I$ is not a radical ideal. For example, if $I=\langle yx,x^2\rangle$, then $C_{y,I} = \langle x\rangle$ and $N_{y,I}=\langle x^2\rangle$, so $\sqrt{C_{y,I}}=\sqrt{N_{y,I}}$. Therefore, $y$ defines a degenerate KMY decomposition of $I$, but the reduced Gr\"obner basis of $I$ involves $y$.
\end{remark}

The final result of this subsection tells us important information about the relationship between nondegenerate KMY decompositions and the height of the ideal. Recall that given an ideal $I\subset R$, we say that $R/I$ is \emph{equidimensional} if $\dim(R/P)=\dim(R/I)$ for all minimal primes $P$ of $I$.

\begin{theorem}\label{theorem: new KR 2.8}
Let $I\subseteq R$ be an ideal such that $R/I$ is equidimensional. Suppose that $I$ has a reduced Gr\"obner basis with respect to a $y$-compatible order $<$ of the form
\[ I = \langle y^{d_1}q_1+r_1, y^{d_2}q_2+r_2,\ldots,y^{d_k}q_k+r_k, h_1,\ldots,h_s\rangle\]

\noindent where $y$ does not divide any term of $q_i$ and $h_j$ for $1 \le i \le k$ and $1 \le j \le s$. If $y$ defines a nondegenerate Knutson-Miller-Yong decomposition of $I$, then \[\height(C_{y,I})=\height(I)=\height(N_{y,I})+1.\]
Furthermore, $R/C_{y,I}$ is equidimensional.
\end{theorem}

\begin{proof}
We first note that since $y$ does not appear in the generating set of $N_{y,I}$, that $\sqrt{N_{y,I}+\langle y\rangle}= \sqrt{N_{y,I}}+\langle y\rangle$. The proof now follows \emph{mutatis mutandis} from \cite[Lemma 2.8]{KR} with $\sqrt{\init_y I}$, $\sqrt{C_{y,I}}$ and $\sqrt{N_{y,I}}$ replacing the ideals $\init_y I$, $C_{y,I}$ and $N_{y,I}$ in that proof, respectively. Furthermore, $R/\sqrt{C_{y,I}}$ is equidimensional from the same proof, which implies that $R/C_{y,I}$ is equidimensional. \qedhere

\end{proof}

\subsection{KMY Decompositions and Toric Ideals of Graphs}

In this section, we consider KMY decompositions for the case where $I=I_G$ is a toric ideal of a graph. We will abuse notation and refer to $y\in E(G)$ as both an edge of $G$ and as an indeterminate in the ring $\mathbb{K}[E(G)]$. We refer the reader to \cite{RTT} for the relevant definitions and background. 

The next theorem follows from \cite[Corollary 3.3]{RTT} on the structure of primitive closed even walks, together with the translation between primitive closed even walks and generators of toric ideals of graphs.

\begin{theorem}\label{theorem: pcew classification}
Let $G$ be a finite simple graph. Suppose that $W$ is a connected subgraph which is the graph of a walk defining a primitive binomial $\Gamma=u-v \in I_G$. Then
\begin{enumerate}[(i)]
    \item $W$ is either an even cycle or contains at least two odd cycles with no overlap in their edges. 
    \item $\Gamma$ has degree at least 2.
    \item If $y^d$ divides $u$, then $d\leq 2$ and $\deg(u)>d$.
    \item If $y|u$, then $y\nmid v$.
\end{enumerate}

\end{theorem}
\begin{proof}
The first three statements follow immediately from the structure of primitive closed even walks described in \cite[Corollary 3.3]{RTT} and \cite[Lemma 3.1]{DAli}. The fourth statement follows from the definition of primitive (if $\Gamma=y(u'-v')$, then $\varphi_G(\Gamma)=\varphi_G(y)\varphi_G(u'-v')=0$ so $u'-v'\in I_G$).
\end{proof}

\begin{lemma}\label{lemma: trivial GVD iff not pcew}
Let $G$ be a finite simple graph and $y\in E(G)$. Then $y$ defines a degenerate KMY decomposition of the toric ideal $I_G$ if and only if $y$ is not part of any primitive closed even walk of $G$. 
\end{lemma}
\begin{proof}
Suppose that $y$ defines a degenerate KMY decomposition of $I_G$. Towards a contradiction, assume that $y$ is part of a primitive closed even walk of $G$. Then there exists a binomial of the form $y^dm-n$ where $y$ does not divide $m$, 
where $1\leq d$ (in fact $d\leq 2$ by Theorem \ref{theorem: pcew classification}) and $\mathrm{deg}(m)\geq 1$, 
with $y^dm\in \init_y(I_G)$. This implies that $C_{y,I}\neq\langle 1\rangle$, so $y$ defining a degenerate KMY decomposition implies that  $\sqrt{C_{y,I}}=\sqrt{N_{y,I}}$. Since $I_G$ is the toric ideal of a graph, and hence is prime, by Proposition \ref{lemma: new KR 2.4} we have $\init_y(I_G) = I_G$. Therefore $y^dm\in I_G$, contradicting the fact that there are no monomials in $I_G$.

For the reverse direction, suppose that $y$ does not divide any term of any primitive binomial of $I_G$. Then $\init_y(g) = g$ for all primitive binomials $g$ (the set of which defines a Gr\"obner basis of $I_G$ with respect to any monomial order). Therefore, $N_{y,I_G} = C_{y,I_G}=I_G$, so that $\sqrt{C_{y,I}}=\sqrt{N_{y,I}}$. Therefore $y$ defines a degenerate KMY decomposition of $I_G$.
\end{proof}

\begin{remark}
By the arguments in the above proof, a degenerate KMY decomposition of the form $C_{y,I}=\langle 1\rangle$ is not possible for $I=I_G$, so degenerate here can only describe the case where $\sqrt{C_{y,I}}=\sqrt{N_{y,I}}$. 
\end{remark}

\noindent

Note that from a graph-theoretic viewpoint, the next two results are easy exercises that follow from the structure on $G$ imposed by the existence of a nondegenerate KMY decomposition. We include these results in order to restate these observations in the language of KMY decompositions.

\begin{lemma}\label{lemma:only in odd}
Let $G$ be a finite simple graph and let $y\in E(G)$. If $G$ is not bipartite and $G\setminus y$ is bipartite 
then $y$ does not belong to any even cycle of $G$. Furthermore, $y$ belongs to every odd cycle of $G$.
\end{lemma}

\begin{proof}
We refer to Figure \ref{fig:bp to nbp} to demonstrate the ideas in this proof. Since $G\setminus y$ is bipartite, there exists a partition of the vertices $V(G\setminus y)$ into the sets $V_A$ and $V_B$. As adding $y$ creates a non-bipartite graph $G$, without loss of generality, we can assume that $y$ must be incident to two vertices $v_{a_1}, v_{a_2} \in V_A$, as shown in Figure \ref{fig:bp to nbp}. Furthermore, as a graph is bipartite if and only if all of its cycles are even, adding $y$ creates at least one odd cycle in $G$, and $y$ must be an edge in all odd cycles of $G$. The lemma now follows by observing  that any path from $v_{a_1}$ to $v_{a_2}$ in $G\setminus y$ must have even length. Therefore, any cycle of $G$ through $y$ (which is comprised of the edge $y$ and a path from $v_{a_1}$ to $v_{a_2}$) must have odd length (in our figure, $y$ is part of a 3-cycle). \qedhere

\begin{figure}[h!]
    \centering
    \includegraphics[scale=1]{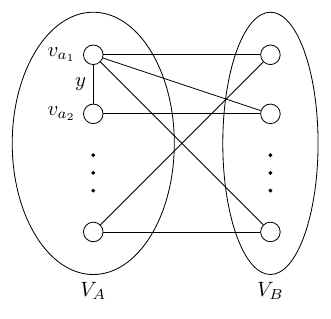}
    \caption{A non-bipartite graph $G$ such that $G\setminus y$ is bipartite.}
    \label{fig:bp to nbp}
\end{figure}
\end{proof}

The next result uses KMY decompositions to detect when a graph deletion of a non-bipartite graph is a bipartite graph. 

\begin{proposition}\label{prop:notBipartiteGVD}
Suppose that $G$ is a non-bipartite finite simple graph and $y\in E(G)$ defines a nondegenerate KMY decomposition of $I_G$. Then $G\setminus y$ is not bipartite.
\end{proposition}
\begin{proof}
Assume towards a contradiction that $G \setminus y$ is bipartite. By Lemma \ref{lemma: trivial GVD iff not pcew}, since $y$ defines a nondegenerate KMY decomposition, the edge $y$ must be part of some primitive closed even walk of $G$. Let us denote the subgraph of $G$ defined by this primitive closed even walk by $W$. By Lemma \ref{lemma:only in odd}, $y$ cannot be in any even cycle of $G$, only odd cycles, and by Theorem \ref{theorem: pcew classification}, $W$ contains at least two odd cycles with no overlap in their edges. But $y$ belongs to all odd cycles of $G$ by Lemma \ref{lemma:only in odd} (otherwise $G \setminus y$ would not be bipartite), a contradiction.
\end{proof}

We conclude this section with a lemma which allows us to iterate the deletion construction alluded to in the previous results. Indeed, starting with a toric ideal of a graph $I_G$, if $y$ defines a KMY decomposition, then $N_{y,I_G}$ is still the toric ideal of a graph. This allows us to compute the height of $I_G$ by induction. 

\begin{lemma}\label{lemma:N ideal deletion ideal}
Let $G$ be a finite simple graph and $<$ be a monomial order on $\mathbb{K}[E(G)]$. If  $e\in E(G)$ then 
$N_{e,I_{G}} = I_{G\setminus e}$. If $M=\{m_1,\ldots, m_r\}$ is the unique minimal monomial generating set of $\init_<(I_G)$, then  $\init_<(I_{G\setminus e})$ is the ideal in $\mathbb{K}[E(G)\setminus e]$ generated by the set of monomials $\{ m_i\in M| e\nmid m_i\}$.
\end{lemma}
\begin{proof}
The first claim follows from \cite[Lemma 3.5]{CDSRVT} since the definition of $N_{e,I_G}$ for KMY decompositions agrees with that for geometric vertex decompositions. For the second claim, let $\{e^{d_1}q_1+r_1,\ldots,e^{d_k}q_k+r_k, h_1,\ldots,h_s\}$ be a Gr\"obner basis of $I_G$ with respect to $<$ as in Definition \ref{def:KMYD}. Then $N_{e,I_G}=\langle h_1,\ldots,h_s\rangle$, which is also equal to $I_{G\setminus e}$ by the above. Furthermore, by \cite[Theorem 2.1]{KMY}, $\{h_1,\ldots, h_s\}$ is a Gr\"obner basis for $N_{e,I_G}$ with respect to the induced monomial order $<$ on $\mathbb{K}[E(G\setminus e)]$. The second claim immediately follows.
\end{proof}

\section{Height Formula Proof}
It is well-known that the height of a toric ideal of graph $G$ can be computed in terms of the number of edges and vertices of $G$. A precise statement of this result appears below in Theorem \ref{theorem: new main}, and a proof due to Villarreal can be found in \cite[Lem. 3.1 and Prop. 3.2]{Villa2}.

In this section, we provide a new proof of this height formula using KMY decompositions. We begin with two auxiliary lemmas before proving the height formula. First we show that when $G$ is a connected graph and $I_G \neq \langle 0 \rangle$, we can always choose an edge $e$ in our graph such that $e$ defines a \emph{nondegenerate} KMY decomposition and is not a bridge (recall that $e$ is a bridge of $G$ if and only if it is not contained in any cycle of $G$).

\begin{lemma} \label{lemma: not trivial or bridge}
Let $G$ be a connected finite simple graph such that $I_G\neq \langle 0\rangle$. Then there exists an edge $e\in E(G)$ which is not a bridge of $G$ and defines a nondegenerate KMY decomposition of $I_G$.
\end{lemma}
\begin{proof}
Since $I_G$ is not the zero ideal, there exists at least one primitive closed even walk $\Gamma$ of $G$. By Theorem \ref{theorem: pcew classification}, any primitive closed even walk includes at least one cycle of $G$, and therefore, at least one edge of the subgraph $\Gamma$ cannot be a bridge of $G$. Furthermore, since $\Gamma$ is primitive, this edge defines a nondegenerate KMY decomposition by Lemma \ref{lemma: trivial GVD iff not pcew}.
\end{proof}

\begin{lemma}\label{lemma: I_G=0}
Let $G$ be a connected finite simple graph. If $I_G=\langle 0\rangle$, then $G$ has no even cycles and at most one odd cycle. 
\end{lemma}
\begin{proof}
Since $I_G=\langle 0\rangle$, $G$ does not have any closed even walks and thus no primitive closed even walks. By Theorem \ref{theorem: pcew classification}, $G$ cannot contain any even cycles. Hence, if $G$ is bipartite, then $G$ must be acyclic and connected, so it must be a tree.
    
If $G$ is not bipartite, then $G$ has at least one odd cycle. Since $G$ has no closed even walks, $G$ cannot have more than one odd cycle. Indeed, if there were two odd cycles, they would either intersect or be disjoint with at least one path connecting them (since $G$ is connected). Tracing either of these subgraphs out would result in a closed even walk. In the first case, starting at a point of intersection, trace out one cycle entirely followed by the second cycle entirely. In the second case, the path connecting the two odd cycles needs to be traversed twice. Either way, the result is a closed even walk, a contradiction.
\end{proof}

\begin{remark}\label{rem: zero cases}
The previous result shows that if $G$ is a connected graph with $I_G=\langle 0\rangle$, then $G$ is either a tree (the bipartite case), or a tree with one edge added to create an odd cycle (the non-bipartite case).
\end{remark}

We now prove the following height formula for the toric ideal of a graph using the methods that we have developed above. This result is known and can be found in Villarreal's work \cite[Lem. 3.1 and Prop. 3.2]{Villa2}. There, a homological approach involving a certain subspace of the cycle space of $G$ is used to compute the height, and while the theme of isolating graph cycles using edge deletion or addition is repeated, the idea of constructively utilizing spanning tree presentation ideals of linear type is different.

\begin{theorem} \label{theorem: new main}
Let $G$ be a connected finite simple graph with $p$ vertices, $q$ edges, and toric ideal $I_G \subset \mathbb{K}[E(G)]$. Then 
\begin{equation*}
    \height(I_G)=
    \begin{cases}
      q-p & \text{if $G$ is not bipartite,}\\
      q-p+1 & \text{if $G$ is bipartite.}
    \end{cases}
  \end{equation*}
\end{theorem}

\begin{proof}
To prove the claim, we use strong induction on $w$, the number of primitive closed even walks of $G$ (that is, the number of primitive binomials of $I_G$ which correspond to closed even walks of $G$).

First, suppose that $w=0$. Then $I_G=\langle 0 \rangle$, and $\height(I_G)=0$. By Remark \ref{rem: zero cases}, when $G$ is bipartite, $G$ is a tree. Thus, $p=q+1$ and $\height(I_G)=0=q-p+1$. By the same remark, when $G$ is not bipartite, $G$ can be constructed by taking a tree (which has $p-1$ edges) and adding an additional edge to the graph, and so $q=(p-1)+1=p$. Therefore, $\height(I_G)=0=q-p$.

Next, let $w\geq 0$ and assume that the statement holds for all graphs with $w$ or fewer primitive closed even walks. Consider a connected graph $G$ with $q$ edges and $p$ vertices and $w+1$ primitive closed even walks. Since $w\geq 0$, we have $I_G \neq \langle 0 \rangle$, and so by Lemma \ref{lemma: not trivial or bridge} there always exists an edge of $G$ that is not a bridge and defines a nondegenerate KMY decomposition of $I_G$. Let us choose $y$ to be such an edge, so that $G_1:=G \setminus y$ is connected, and $w_1 < w$ where $w_1$ is the number of primitive closed even walks of $G_1$.

Thus by our inductive hypothesis, Lemma \ref{lemma:N ideal deletion ideal} and Theorem \ref{theorem: new KR 2.8} (recall that toric ideals of graphs are prime and hence equidimensional), it follows that   
\begin{equation*}
        \height(I_G)=\height(I_{G\setminus y}) + 1= 
    \begin{cases}
      ((q-1)-p)+1 & \text{if $G_1$ is not bipartite,}\\
      ((q-1)-p+1)+1 & \text{if $G_1$ is bipartite,}
    \end{cases}
\end{equation*}
noting that in an edge deletion, only the edge is removed and not the vertices, so $V(G)=V(G\setminus y)$. Finally, the conditions need to be stated in terms of whether $G$ is bipartite. We note that if $G$ is not bipartite, then by Proposition \ref{prop:notBipartiteGVD}, $G_1$ is also not bipartite, so the first condition above can be restated as ``if $G$ is not bipartite". On the other hand, if $G$ is bipartite, then $G_1$ must also be bipartite (the chromatic number can only decrease when deleting edges). 

Therefore, $\height(I_G) = q-p$ if $G$ is not bipartite, and $\height(I_G)=q-p+1$ if $G$ is bipartite, proving the result.
\end{proof}

\begin{remark}\label{remark: deletion sequence}
Let $e_1,\ldots,e_k$ be any sequence of edges and define $G_i:=G\setminus\{e_1,\ldots, e_i\}$ for $1\leq i\leq k$ (that is, the graph formed by deleting $e_1,\ldots,e_i$ from $G$). We set $G_0:=G$. Each $e_i$ defines a KMY decomposition of $I_{G_{i-1}}$ for $1\leq i\leq k$, and if $I_{G_k}=\langle 0\rangle$, then the theorem shows that the height of $I_G$ is also equal to the number of nondegenerate KMY decompositions in the sequence. This follows from the fact that $\height(I_{G_{i-1}})=\height(I_{G_i})$ if $e_i$ defines a degenerate KMY decomposition. Some care must be taken if any of the edges are a bridge, since the graph deletion would be disconnected, but the result still holds since $\height(I_{H_1\sqcup H_2}) = \height(I_{H_1})+\height(I_{H_2})$.
\end{remark}

\section{A bound on $\chi(G)$ using Gr\"obner degeneration}

In this section, we will provide a bound on the chromatic number of a graph $G$ in terms of a Gr\"obner basis of $I_G$. This bound follows from the simple observation that computing $\init_<(I_{G\setminus e})$ with respect to some monomial order $<$ can easily be computed from $\init_<(I_G)$ using KMY decompositions. At the same time, deleting an edge reduces the chromatic number by at most 1. This fact will be very  clear to experts, but we include the brief argument here for completeness.

\begin{lemma}\label{lem: chromatic_reduction}
Let $G$ be finite simple graph and $e\in E(G)$. Then \[0\leq \chi(G) - \chi(G\setminus e) \leq 1.\]
\end{lemma}
\begin{proof}
Let $e=\{u,v\}$, and suppose that $\mathcal{C}$ is a (minimal) vertex  coloring of $G\setminus e$. If $u$ and $v$ have different colors in $\mathcal{C}$, then $\mathcal{C}$ remains a coloring of $G$ after adding the edge $e$ to the graph $G\setminus e$. That is $\chi(G)=\chi(G\setminus e)$. On the other hand, if $u$ and $v$ have the same color in $\mathcal{C}$, then we define a new coloring $\mathcal{C}'$ which equals $\mathcal{C}$ everywhere except at $u$. Color $u$ with a new color not included in $\mathcal{C}$. Then $\mathcal{C}'$ is a $\chi(G\setminus e)+1$ coloring of $G$, proving the result.
\end{proof}

This result allows us to bound $\chi(G)$ by deleting edges of $G$ until the resulting graph has an associated toric ideal which is the zero ideal (ie. until there are no other primitive closed even walks left). We start with an example to demonstrate this approach. 

\begin{example}
Let $G$ be the graph defined by two 4-cycles glued at an edge. 

\begin{figure}[!h]
\centering
\begin{tikzpicture}[shorten >=1pt, auto,scale=1,
   node_style/.style={circle,draw=black,scale=0.25},
   edge_style/.style={draw=black}]
    \node[node_style] (v1) at (-2,2) {};
    \node[node_style] (v2) at (0,2) {};
    \node[node_style] (v3) at (2,2) {};
    \node[node_style] (v4) at (2,0) {};
    \node[node_style] (v5) at (0,0) {};
    \node[node_style] (v6) at (-2,0) {};
    \draw[edge_style]  (v1) edge node{$e_1$} (v2);
    \draw[edge_style]  (v2) edge node{$e_2$} (v3);
    \draw[edge_style]  (v3) edge node{$e_3$} (v4);
    \draw[edge_style]  (v4) edge node{$e_4$} (v5);
    \draw[edge_style]  (v5) edge node{$e_5$} (v6);
    \draw[edge_style]  (v6) edge node{$e_6$} (v1);
    \draw[edge_style]  (v2) edge node{$e_7$} (v5);
\end{tikzpicture}
\label{fig:2 4-cycles}
\end{figure}

The toric ideal $I_G$ is generated by the primitive binomials corresponding to the two 4-cycles of the graph:  \[
I_G = \langle e_1e_5-e_6e_7, e_2e_4-e_3e_7 \rangle \subseteq \mathbb{K}[e_1, e_2, e_3, e_4, e_5, e_6, e_7].\]

\noindent Let $H_1=G\setminus e_6$. Then it is not difficult to compute that \[I_{H_1}=\langle  e_2e_4-e_3e_7\rangle. \]
At the same time, by Lemma \ref{lem: chromatic_reduction}, $\chi(G)-1\leq \chi(H_1)\leq \chi(G)$. Continuing, we can delete $e_3$ to get the graph $H_2=H_1\setminus e_3$. Then $I_{H_2} = \langle 0\rangle$. On the other hand, the deletion drops the chromatic number by at most one. 

Now that $I_{H_2}=\langle 0\rangle$, we can conclude that $\chi(H_2)\leq 3$ by Lemma \ref{lemma: I_G=0}. Indeed, the chromatic number of a single cycle is at most 3, and the chromatic number of a tree is at most 2. Using this, it is straightforward to show that a graph satisfying the conditions of Lemma \ref{lemma: I_G=0} has chromatic number bounded above by 3. Therefore, \[\chi(G)\leq \chi(H_2)+2 = 5.\qedhere\]
\end{example}

Notice that in the last example, the same bound could have been achieved by working with $\init_<(I_G)$ instead of $I_G$. Indeed, a generator of $I_G$ does not appear in $I_{H_1}$ if it contained the variable $e_6$. The same result can be achieved if we considered any lexicographic order where $e_6>e_3>f$ where $f\in \{e_1,e_2,e_4,e_5,e_7\}$. Then 
\[\init_<(I_G) = \langle e_6e_7,e_3e_7,e_2e_4e_6\rangle \hspace{2mm}\longrightarrow \hspace{2mm} \init_<(I_{H_1}) = \langle e_3e_7\rangle \hspace{2mm}\longrightarrow\hspace{2mm} \init_<(I_{H_2}) = \langle 0\rangle.\]

Finding a sequence of edges which result in the zero ideal can now be phrased as finding a minimal ideal of indeterminates which contains $\init_<(I_G)$. In the previous example, $\init_<(I_G) \subset \langle e_6,e_3\rangle$, so $\chi(G)\leq |\{e_6,e_3\}|+3$. 

By \emph{minimal}, we mean a subset $\mathcal{E}$ of the variables for which the property holds, and such that no proper subset $\mathcal{E}'\subset \mathcal{E}$ also satisfies the property. Our initial choice of $\mathcal{E}$ however need not be a set with the smallest cardinality among the sets that satisfy the property. For instance, if $I$ is an ideal such that $\init_<(I) = \langle xy,xz \rangle$, then we can take $\mathcal{E}=\{x\}$ or $\mathcal{E}=\{y,z\}$. Both are minimal (even though one obviously leads to a better bound). This leads to the following result.

\begin{theorem}\label{thm: chromatic_bound}
Let $G$ be a connected finite simple graph defining the toric ideal $I_G \subset R=\mathbb{K}[e_1,\ldots,e_n]$ and let $<$ be a monomial order on $R$. Suppose that $\mathcal{E}\subseteq \{e_1,\ldots,e_n\}$ is a (possibly empty) minimal set of variables such that $\init_<(I_G)\subseteq \langle \mathcal{E}\rangle$. Then 

\[\chi(G)\leq |\mathcal{E}| +3.\]
\end{theorem}
\begin{proof}

We proceed by induction on $r=|\mathcal{E}|$. If $r=0$, then $I_G=\langle 0\rangle$. By Lemma \ref{lemma: I_G=0}, $I_G$ has no even cycles and at most one odd cycle, so $\mathcal{E}=\emptyset$, $\langle 0 \rangle = \init_<(I_G)\subseteq \langle \mathcal{E}\rangle = \langle 0\rangle$, and $\chi(G)\leq |\mathcal{E}|+3 = 3$. Assume the result is true for $r=k$ and consider a graph $G$ such that $|\mathcal{E}|=k+1$. Take any variable $e\in \mathcal{E}$. By definition of $\mathcal{E}$, the edge $e$ appears in at least one primitive closed even walk of $G$. 

By Lemma \ref{lemma:N ideal deletion ideal}, $N_{e,I_{G}} = I_{G\setminus e}$, and $\init_<(I_{G\setminus e})$ is generated by those monomials in $\init_<(I_G)$ which are not divisible by $e$. In particular, $\init_<(I_{G\setminus e}) \subseteq \langle\mathcal{E}\setminus e \rangle$ and $\mathcal{E}\setminus e$ is minimal. By induction, \[\chi(G\setminus e)\leq |\mathcal{E}\setminus e| +3 = |\mathcal{E}|+2.\]

On the other hand, \[\chi(G)\leq \chi(G\setminus e)+1\leq |\mathcal{E}|+2+1 = |\mathcal{E}|+3,\] as required.
\end{proof}

\begin{remark}\label{rem: height}
As part of an undergraduate thesis project at Virginia State University, the first author and Tiandra Prather made the following observation. Let $\cap_i P_i$ be the primary decomposition of $\init_<(I_G)$. First notice  that each minimal prime of $\init_<(I_G)$ defines one possible choice of $\mathcal{E}$ because $\sqrt{P_i}$ is an ideal of indeterminates and $\init_<(I_G)\subset\sqrt{P_i}$.  Since $\height(I_G) = \height(\init_<(I_G)) = \min_i\{\height(P_i)\}$, there is some choice of $P_i$ such that $\height(I_G) = |\mathcal{E}|$ since $\height(P_i)=\height(\sqrt{P_i})$ which equals the number of indeterminates defining it. In fact, by Kalkbrener and Sturmfels \cite{KS}, since $I_G$ is prime, $\init_<(I_G)$ is equidimensional, so $\height(P_i) = |\mathcal{E}|$ for every choice of minimal $\sqrt{P_i}$.
\end{remark}

\begin{remark}
The initial ideal $\init_<(I_G)$ is the edge ideal of a hypergraph $X$ if $\init_<(I_G)$ is squarefree (where each monomial defines an edge of $X$). In this case, $\mathcal{E}$ is a choice of a minimal vertex cover of $X$. In particular, choosing $\mathcal{E}$ with cardinality equal to the vertex covering number of $X$ provides the best possible bound on $\chi(G)$ in Theorem \ref{thm: chromatic_bound}.
\end{remark}

In light of Remark \ref{rem: height} and Theorem \ref{theorem: new main}, we get that $\chi(G)\leq |E(G)|-|V(G)|+3$ for non-bipartite graphs, underlining the fact that this current algebraic bound is coarse. Graphs for which $|E(G)|-|V(G)|$ is small enough do not have chromatic numbers which are difficult to compute. Nonetheless, we contrast this bound to the well-known bound on $\chi(G)$ of the maximum degree $\Delta_G$ of a vertex of $G$ plus $1$. The next three examples provide instances where our bound is better than, equal to, and worse than $\Delta_G+1$. 

\begin{example}

Let $G$ be the extended bow-tie graph pictured below. 
    \begin{figure}[h!]
\centering
    \begin{tikzpicture}[scale=0.5]
      
      \draw (0,3) -- (0,9)node[midway,  left] {$e_1$};
      \draw (4.5,6) -- (0,9) node[midway, above] {$e_3$};
      \draw (0,3) -- (4.5,6) node[midway, below] {$e_2$};
      \draw (4.5,6) -- (7.5,6) node[midway, above] {$e_4$};
      \draw (7.5,6) -- (10.5,6) node[midway, above] {$e_5$};
      \draw (10.5,6) -- (14,3) node[midway, below] {$e_6$};
      \draw (10.5,6) -- (14,9) node[midway, above] {$e_7$};
      \draw (14,3) -- (14,9) node[midway, right] {$e_8$};

      \fill[fill=white,draw=black] (0,3) circle (.1) node[left]{};
      \fill[fill=white,draw=black] (0,9) circle (.1) node[left]{};
      \fill[fill=white,draw=black] (4.5,6) circle (.1) node[below]{};
      \fill[fill=white,draw=black] (7.5,6) circle (.1) node[below]{};
      \fill[fill=white,draw=black] (10.5,6) circle (.1) node[below]{};
      \fill[fill=white,draw=black] (14,3) circle (.1) node[right]{};
      \fill[fill=white,draw=black] (14,9) circle (.1) node[right]{};
     
    \end{tikzpicture}
    \label{fig:2triangles}
\end{figure}

Then $I_G = \langle e_1e_4^2e_6e_7 - e_2e_3e_5^2e_8 \rangle$, and picking any lexicographic order where $e_1$ has the greatest weight will yield \[\init_<(I_G) = \langle e_1e_4^2e_6e_7\rangle. \]
Then $\{e_1\}$ is a set which satisfies the assumptions of the theorem. Therefore, $|\mathcal{E}|=1$ and  \[\chi(G) \leq 4.\] By inspection we see that $\chi(G)=3$ and $\Delta_G+1=4$.
\end{example}

When $I_G$ is principal, we know from Theorem \ref{theorem: pcew classification} that $G$ has at most two cycles. Showing that $\chi(G)\leq 4$ by deleting one edge from each cycle in this case is an easy exercise. As a parallel demonstration of the algebraic approach, we also provide a Gr\"obner proof of this fact.

\begin{corollary}\label{cor: principal}
Let $G$ be a finite simple graph such that $I_G$ is a principal ideal. Then $\chi(G)\leq 4$.
\end{corollary}
\begin{proof}
If $I_G=\langle 0\rangle$, then the result holds by Lemma \ref{lemma: I_G=0}. Otherwise, $I_G = \langle f\rangle$ for some non-zero $f\in \mathbb{K}[E(G)]$, and for any lexicographic order $<$ on $\mathbb{K}[E(G)]$, $\init_<(I_G) = \langle \init_<(f)\rangle$. Clearly $\langle \init_<(f)\rangle \subset \langle e_i\rangle$ for any $e_i$ which divides $\init_<(f)$. Then $|\mathcal{E}|=1$, and the result follows. 
\end{proof}

\begin{example}
Let $G=K_4$, the complete graph on 4 vertices, pictured below:

\begin{figure}[ht!]
\centering
    \begin{tikzpicture}[scale=0.6]
      
      \draw (-4,0) -- (0,6)node[midway,  left] {$e_1$};
      \draw (-4,0) -- (0,2.5) node[midway, above] {$e_2$};
      \draw (0,2.5) -- (0,6) node[midway, right] {$e_3$};
      \draw (4,0) -- (0,6) node[midway, right] {$e_4$};
      \draw (0,2.5) -- (4,0) node[midway, above] {$e_5$};
      \draw (-4,0) -- (4,0) node[midway, below] {$e_6$};

      \fill[fill=white,draw=black] (-4,0) circle (.1) node[left]{};
      \fill[fill=white,draw=black] (4,0) circle (.1) node[left]{};
      \fill[fill=white,draw=black] (0,2.5) circle (.1) node[below]{};
      \fill[fill=white,draw=black] (0,6) circle (.1) node[below]{};
     
    \end{tikzpicture}
    \label{fig:2triangles}
\end{figure}

Using Macaulay2, we can check that the generating set $I_G=\langle e_1e_5-e_3e_6, e_2e_4-e_3e_6\rangle$ is a Gr\"obner basis with respect to the lexicographic order $e_1>e_2>\ldots>e_6$. Then \[\init_<(I_G) = \langle e_1e_5,e_2e_4\rangle \subset \langle e_1,e_2\rangle,\] so $|\mathcal{E}|=2$ and $\chi(G)\leq 5$, which differs from the actual chromatic number by 1. In this case, $\Delta_G+1=4$ provides a better bound.
\end{example}  

\begin{example}
The bow-tie graph $G$ below has a principal toric ideal $I_G$, so by Corollary \ref{cor: principal}, $\chi(G)\leq 4$.  
    \begin{figure}[h!]
\centering
    \begin{tikzpicture}[scale=0.5]
      
      \draw (0,3) -- (0,9)node[midway,  left] {$e_1$};
      \draw (4.5,6) -- (0,9) node[midway, above] {$e_3$};
      \draw (0,3) -- (4.5,6) node[midway, below] {$e_2$};
      \draw (4.5,6) -- (9,3) node[midway, below] {$e_4$};
      \draw (4.5,6) -- (9,9) node[midway, left] {$e_5$};
      \draw (9,3) -- (9,9) node[midway, right] {$e_6$};

      \fill[fill=white,draw=black] (0,3) circle (.1) node[left]{};
      \fill[fill=white,draw=black] (0,9) circle (.1) node[left]{};
      \fill[fill=white,draw=black] (4.5,6) circle (.1) node[below]{};
      \fill[fill=white,draw=black] (9,3) circle (.1) node[right]{};
      \fill[fill=white,draw=black] (9,9) circle (.1) node[right]{};
     
    \end{tikzpicture}
\end{figure}

\noindent Its actual chromatic number is 3, and $\Delta_G+1=5$, so Theorem \ref{thm: chromatic_bound} provides a better bound in this case.
\end{example}

\end{document}